\documentclass[11pt]{amsart}
\pdfoutput=1

\usepackage{amsmath}
\usepackage{amssymb}
\usepackage{amsthm}
\usepackage{multicol}
\usepackage[foot]{amsaddr}
\usepackage[margin=1in]{geometry}
\usepackage[hidelinks]{hyperref}
\usepackage{bm}
\usepackage{mathtools}
\usepackage{lmodern}
\usepackage[T1]{fontenc}
\usepackage{microtype}

\DeclareMathOperator{\ddeg}{ddeg}

\def\d {\operatorname{d}}
\def\dh {\operatorname{dh}}

\newcommand{\ca}{\mathcal}
\newcommand{\bb}{\mathbb}

\newcommand{\pconv}{\rightsquigarrow}
\newcommand{\prece}{\preccurlyeq}

\newcommand{\fd}{\mathfrak{d}}

\newcommand{\x}{\times}

\newcommand{\ges}{\geqslant}
\newcommand{\les}{\leqslant}

\DeclareFontFamily{U}{fsy}{}
\DeclareFontShape{U}{fsy}{m}{n}{<->s*[.9]psyr}{}
\DeclareSymbolFont{der@m}{U}{fsy}{m}{n}
\DeclareMathSymbol{\der}{\mathord}{der@m}{182}
\DeclareFontFamily{OMS}{smallo}{}
\DeclareFontShape{OMS}{smallo}{m}{n}{<->s*[.65]cmsy10}{}
\DeclareSymbolFont{smallo@m}{OMS}{smallo}{m}{n}
\DeclareMathSymbol{\cao}{\mathord}{smallo@m}{79}

\newtheorem{thm}{Theorem}[section]
\newtheorem*{thm*}{Theorem}
\newtheorem{lem}[thm]{Lemma}
\newtheorem{prop}[thm]{Proposition}
\newtheorem{cor}[thm]{Corollary}

\theoremstyle{definition}

\newtheorem*{defn}{Definition}

\theoremstyle{remark}

\makeatletter\@addtoreset{case}{thm}\makeatother

\usepackage{etoolbox}

\makeatletter
\patchcmd{\@startsection}
  {\@afterindenttrue}
  {\@afterindentfalse}
  {}{}
\makeatother

\title{On the uniqueness of maximal immediate extensions of valued differential fields}
\author{Lou van den Dries}
\email{\href{mailto:vddries@math.uiuc.edu}{vddries@math.uiuc.edu}}
\author{Nigel Pynn-Coates}
\address{Department of Mathematics, University of Illinois at Urbana--Champaign, Urbana, IL 61801, USA}
\email{\href{mailto:pynncoa2@illinois.edu}{pynncoa2@illinois.edu}}
\thanks{The second author was partially supported by an NSERC Postgraduate Scholarship.}

\begin{document}

\begin{abstract}
So far there exist just a few results about the uniqueness of maximal immediate valued differential field extensions and about the relationship between differential-algebraic maximality and differential-henselianity; see \cite[Chapter 7]{adamtt}. 
We remove here the assumption of monotonicity in these results but replace it with the assumption that the value group is the union of its convex subgroups of finite (archimedean) rank.
We also show the existence and uniqueness of differential-henselizations of asymptotic fields with such a value group.
\end{abstract}

\maketitle

\section{Introduction}

This paper is concerned with various uniqueness properties of immediate differential-henselian extensions of valued differential fields.
First, some notations and conventions, which we keep as close as possible to \cite{adamtt}.
A \emph{valued differential field} is a field with both a valuation and a derivation on it, of equicharacteristic 0 (that is, the residue field has characteristic 0).
Throughout, $K$ will be a valued differential field with value group $\Gamma$, (surjective) valuation $v \colon K^\x \to \Gamma$, 
valuation ring $\ca O=\{a\in K:\ v(a) \ges 0\}$, maximal ideal
$\cao=\{a: v(a) > 0\}$ of $\ca O$, and residue field
$\bm k=\ca O/\cao$. We let
$\der$ be the derivation of $K$ and $C=\{a\in K:\ \der(a)=0\}$ the constant field of $K$. When the derivation is clear from context, we often write $a'$ for $\der(a)$.
Given also a valued differential field $L$, we let $\Gamma_L$ be its value group, $\bm k_L$ be its residue field, etc.

Our primary assumption is that $K$ has \emph{small derivation}, i.e., $\der \cao \subseteq \cao$.
It follows that $\der$ is continuous with respect to the valuation topology on $K$ (see \cite[Lemma~4.4.6]{adamtt}) and that $\der\ca O \subseteq \ca O$ (see \cite[Lemma~4.4.2]{adamtt}), so $\der$ induces a derivation on $\bm k$.
Throughout, $\bm k$ will be construed as a differential field with this derivation. A differential field $E$ of characteristic 0 is said to be \emph{linearly surjective} if for all $a_0, \dots, a_r \in E$, $a_r \neq 0$, there is $y \in E$ such that $1+a_0y+a_1y'+\dots+a_ry^{(r)}=0$.
Note that if $\bm k$ is linearly surjective, then the induced derivation on $\bm k$ is nontrivial.

In this paper, the word \emph{extension} abbreviates {\em valued differential field extension with small derivation}.
Similarly, \emph{embedding} will mean {\em valued differential field embedding into a valued differential field with small derivation}.
If $L$ is an extension of $K$, then we consider $\bm k$ as a differential subfield of $\bm k_L$ and $\Gamma$ as a subgroup of $\Gamma_L$ in the usual way.
An extension $L$ of $K$ is \emph{immediate} if $\Gamma_L=\Gamma$ and $\bm k_L=\bm k$.
We say that $K$ is \emph{maximal} if it has no proper immediate extensions.
We say that $K$ is \emph{differential-algebraically maximal} (\emph{$\d$-algebraically maximal} for short) if it has no proper $\d$-algebraic immediate extensions.
Maximal immediate extensions of $K$ exist by Zorn, and are spherically complete by \cite{maximext}.
Likewise, $\d$-algebraically maximal $\d$-algebraic immediate extensions of $K$ exist.

The authors of \cite{maximext} conjecture that if $\bm k$ is linearly surjective, then all maximal immediate extensions of $K$ are isomorphic over $K$.
Similarly, we expect that if $\bm k$ is linearly surjective, then all $\d$-algebraically maximal $\d$-algebraic immediate extensions of $K$ are isomorphic over $K$.
Both conjectures hold when $K$ is \emph{monotone}, i.e., $v(a') \ges v(a)$ for all $a \in \cao$: \cite[Theorem~7.4.3]{adamtt};
for monotone $K$ with {\em many constants}, that is, $v(C^\x)=\Gamma$, this is due to Scanlon \cite{scanlon}.
We are primarily interested in the non-monotone case, and prove these conjectures when $\Gamma$ has \emph{finite {\rm(}archimedean{\rm)} rank}, that is, $\Gamma$ has only finitely many convex subgroups.
More generally:

\begin{thm}\label{maximmedfinrankunion}
If $\bm k$ is linearly surjective and $\Gamma$ is the union of its finite rank convex subgroups,
then any two maximal immediate extensions of $K$ are isomorphic over $K$, and any two $\d$-algebraically maximal $\d$-algebraic immediate extensions of $K$ are isomorphic over $K$.
\end{thm}

\noindent
To prove this, we isolate a property, the \emph{differential-henselian configuration property}, used implicitly in \cite{adamtt}, and do the proof in two steps.
The first step is completely new and shows in \S\ref{sec:dhc} that $\Gamma$ as in the theorem has the differential-henselian configuration property.
The second shows in the same way as in the monotone case that the conclusion of the theorem holds whenever $\bm k$ is linearly surjective
and $\Gamma$ has this property; this is done in \S\ref{sec:mainresults}.

We also use the differential-henselian configuration property to prove the existence and uniqueness of differential-henselizations.
First we recall the definition of differential-henselianity from \cite{adamtt}.
Let $K\{Y\}=K[Y, Y', Y'', \dots]$ be the differential polynomial ring over $K$.
For $P \in K\{Y\}$, $v(P)$ is the minimum valuation of the coefficients of $P$ and $P_d$ is the homogeneous part of $P$ of degree $d$.
\begin{defn}
$K$ is \emph{differential-henselian} (\emph{$\d$-henselian} for short) if:
\begin{enumerate}
\item $\bm k$ is linearly surjective, and
\item for all $P \in \ca O\{Y\}$, if $v(P_0)>0$ and $v(P_1)=0$, then there is $y \in \cao$ with $P(y)=0$.
\end{enumerate}
\end{defn}

\noindent
The notion of $\d$-henselianity is connected to $\d$-algebraic maximality: if $K$ is $\d$-algebraically maximal and $\bm k$ is linearly surjective, then $K$ is $\d$-henselian, by \cite[Theorem~7.0.1]{adamtt}.
The authors of that book also proved a partial converse \cite[Theorem~7.0.3]{adamtt}: if $K$ is monotone, $\d$-henselian, and has \emph{few constants} (i.e., $C \subseteq \ca O$), then $K$ is $\d$-algebraically maximal.
They asked whether the monotonicity assumption can be dropped. 
The following is a partial result in that direction:
\begin{thm}
If $K$ is $\d$-henselian, has few constants, and $\Gamma$ is the union of its finite rank convex subgroups,
then $K$ is $\d$-algebraically maximal.
\end{thm}

\noindent
If $\bm k$ is linearly surjective, then $K$ has an immediate $\d$-henselian extension: just take any immediate extension of $K$ that is maximal (or $\d$-algebraically maximal).
Conversely, if $K$ has an immediate $\d$-henselian extension, then $\bm k$ must be linearly surjective.
In analogy with the henselization of a valued field, we propose: 

\begin{defn}
A \emph{differential-henselization} (\emph{$\d$-henselization} for short) of $K$ is an immediate $\d$-henselian extension of $K$ that embeds over $K$ into every immediate $\d$-henselian extension of $K$.
\end{defn}

\noindent
Now, we turn to the setting of asymptotic fields.
We recall from \cite[\S9.1]{adamtt} that $K$ is said to be \emph{asymptotic} if, for all $f, g \in \cao \setminus\{0\}$,
$$v(f)>v(g)\ \iff\ v(f')>v(g').$$
Note that if $K$ is asymptotic, then it has few constants.
Conversely, if $K$ is $\d$-henselian and has few constants, then it is asymptotic by \cite[Lemmas 9.1.1 and 7.1.8]{adamtt}.

\begin{thm}
If $K$ is asymptotic, $\bm k$ is linearly surjective, and $\Gamma$ is the union of its finite rank convex subgroups, 
then $K$ has, up to isomorphism over $K$, a unique $\d$-henselization.
\end{thm}

\noindent
If $\der$ is trivial, then the results above, with no assumptions on $\Gamma$, are given by the analogous results for valued fields.
If $\Gamma=\{0\}$ (i.e., the valuation is trivial), then $K \cong \bm k$ as differential fields, and $K$ is the unique maximal immediate extension of $K$.

To review the standing assumptions, $K$ is throughout a valued differential field (of equicharacteristic 0) with small derivation.
We henceforth assume that both the derivation and valuation are nontrivial.
Additional assumptions on $K$, $\Gamma$, or $\bm k$ will be specified where needed.
In particular, we sometimes assume that the induced derivation on $\bm k$ is nontrivial.
Note that when this is the case, all extensions of $K$ are \emph{strict}, as defined in \cite{maximext}.

Finally, we often use results from \cite{adamtt}, especially from Sections 2.2, 3.2, 4.3, 4.4, 4.5, 6.1, 6.6, 6.8, 6.9, 7.1, 7.2, and 7.3, but we give precise references. 

The present paper, like
\cite[Chapters 6,7]{adamtt} and \cite{maximext}, is meant to contribute to an emerging theory of valued {\em differential\/} fields, in analogy with the theory of valued fields. One contrast to keep in mind is that small valued differential fields have typically
rather large (infinite rank) value groups. 

\subsection*{Additional notations and conventions}
We let $m, n, r, d \in \bb N = \{0, 1, 2, \dots\}$.
If $\Delta$ is an ordered abelian group, we let $\Delta^> = \{\delta \in \Delta : \delta > 0\}$.
If $E$ is a field, we let $E^\x = E \setminus \{0\}$.
If $E$ is a valued field and its valuation $v_E$ is clear from the context, then we define for $a$, $b \in E^\x$:
$$\begin{array}{lc}
a \prece b\ \Leftrightarrow\ v_E(a)\ges v_E(b),\qquad a \prec b\ \Leftrightarrow\ v_E(a)> v_E(b),\\
a \asymp b\ \Leftrightarrow\ v_E(a)=v_E(b),\qquad  
  a\sim b\ \Leftrightarrow\ a-b \prec b.
\end{array}$$
We abbreviate ``pseudocauchy sequence'' by \emph{pc-sequence} and use facts about them frequently; see \cite[\S2.2 and \S3.2]{adamtt}.
In the differential setting, see \cite[\S4.4, \S6.8, and \S6.9]{adamtt} for the notions of a pc-sequence being of \emph{$\d$-algebraic type} (equivalently, having a \emph{minimal differential polynomial}) and being of \emph{$\d$-transcendental type}.

As to differential polynomials, let $E$ be a differential field (of characteristic 0) and set $E\{Y\}^{\neq} \coloneqq E\{Y\}\setminus\{0\}$.
The \emph{order} of $P \in E\{Y\}$ is the least $r$ such that $P \in E[Y, Y', \dots, Y^{(r)}]$.
Let $r_P$ be the order of $P$. Then the \emph{degree} of $P$ is its total degree as an element of $E[Y, Y', \dots, Y^{(r_P)}]$, denoted by $\deg P$.
Let $s_P$ be the degree of $P$ in $Y^{(r_P)}$ and $t_P=\deg P$.
Then the \emph{complexity} of $P$ is the ordered triple $(r_P, s_P, t_P)$ and is denoted by $c(P)$.
We often use the \emph{multiplicative conjugate} $P_{\x a}=P(aY)$ and the \emph{additive conjugate} $P_{+a}=P(a+Y)$, for $a \in E$.
For more on such conjugation, see \cite[\S4.3]{adamtt}.

For a valued differential field $E$ we extend its valuation $v_E$ to $E\{Y\}$ by 
$$v_E(P)\ \coloneqq\ \min \{v_E(a):\ a \text{ is a coefficient of }P\},$$ and we extend the binary relations
$\prece$, $\prec$, $\asymp$, and $\sim$ on $E$ to $E\{Y\}$ accordingly.

\section{Preliminaries}

In this section, $P \in K\{Y\}^{\neq}$, and 
$(a_\rho)$ is a pc-sequence in $K$ with $\gamma_\rho \coloneqq v(a_{\rho+1}-a_\rho)$; here and later $\rho+1$ denotes the immediate successor of $\rho$ in the well-ordered set of indices.
Recall from \cite[\S6.6]{adamtt} the notion of the \emph{dominant part} of $P$:
To any $P$ we associate $\fd_P \in K^{\x}$ with $\fd_P \asymp P$ such that $\fd_P = \fd_Q$ for all $Q \in K\{Y\}^{\neq}$ with $Q \sim P$. Then
$\fd_P^{-1} P\asymp 1$, so we can define the \emph{dominant part} $D_P\in \bm k\{Y\}^{\neq}$ of $P$ to be the image of $\fd_P^{-1}P$ in $\bm k\{Y\}$ and 
the \emph{dominant degree} of $P$ to be $\ddeg P \coloneqq \deg D_P$.
We refer to \cite[\S6.6]{adamtt} for basic properties of these notions.
For later use we show that the condition
$\ddeg P \ges 1$ is necessary for the existence of a zero $f\prece 1$ of $P$ in an extension of $K$: 

\begin{lem}\label{ddegroot}
Let $g \in K^\times$, and suppose $P(f)=0$ for some $f \prece g$ in some extension of $K$. Then $\ddeg P_{\x g} \ges 1$.
\end{lem}
\begin{proof}
Let $L$ be an extension of $K$ and suppose $f \in L$, $f \preccurlyeq g$, and $P(f)=0$. Then $f=ag$ for some $a \in L$ with $a \preccurlyeq 1$.
Letting $Q=P_{\x g}$, we have $Q(a)=0$, so $D_Q(\bar{a})=0$. Hence, $\ddeg P_{\times g} \ges 1$.
\end{proof}

\noindent
For $\gamma\in \Gamma$  we have by \cite[\S6.6]{adamtt} that
for all $g\in K^\times$ with $vg=\gamma$,
$$\ddeg_{\ges \gamma}P\  \coloneqq \ \max\{\ddeg P_{\times f}:\ f\in K^\times,\ vf\ges \gamma\}\ =\ \ddeg P_{\times g}.$$

\subsection*{Dominant degree in a cut}
Here we define the crucial notion of ``dominant degree in a cut'' and record some basic properties.
This is a special case of the ``newton degree in a cut'' in \cite{maximext}.

\begin{lem}
There is an index $\rho_0$ and a number $d\big(P, (a_\rho)\big) \in \bb N$ such that for all $\rho>\rho_0$,
$$\ddeg_{\ges \gamma_\rho} P_{+a_\rho}\ =\ d\big(P, (a_\rho)\big).$$
Whenever $(b_\sigma)$ is a pc-sequence in $K$ equivalent to $(a_\rho)$, we have $d\big(P,(a_\rho)\big)=d\big(P,(b_\sigma)\big)$.
\end{lem}
\begin{proof}
Take $\rho_0$ such that for all $\rho'>\rho\ges\rho_0$, we have $\gamma_{\rho'}>\gamma_\rho$ and $\gamma_\rho \in \Gamma$.
Then $$\ddeg_{\ges\gamma_{\rho'}} P_{+a_{\rho'}}\ \les\ \ddeg_{\ges\gamma_{\rho}} P_{+a_{\rho}}\ \text{ for all }\rho'>\rho\ges\rho_0$$
by \cite[Corollary~6.6.12]{adamtt}.
This gives the existence of $d\big(P,(a_\rho)\big)$. Set $d=d\big(P,(a_\rho)\big)$.
We can assume $\rho_0$ to be so large that $\ddeg_{\ges\gamma_\rho}P_{+a_\rho}=d$ for all $\rho\ges\rho_0$.
Let $(b_{\sigma})$ be a pc-sequence in $K$ equivalent to $(a_\rho)$, and set $\beta_\sigma=v(b_{\sigma+1}-b_\sigma)$.
Take an index $\sigma_0$ and $e \in \bb N$ so that  $\beta_\sigma \in \Gamma$
and $\ddeg_{\ges \beta_\sigma} P_{+b_\sigma}=e$ for all $\sigma \ges \sigma_0$.
By \cite[Lemma~2.2.17]{adamtt}, we can further arrange that $b_\sigma-a_\rho \prec a_\rho-a_{\rho_0}$ and $\beta_\sigma \ges \gamma_{\rho_0}$ for all $\rho>\rho_0$ and $\sigma>\sigma_0$.
Then for $\sigma>\sigma_0$ we have $v(b_\sigma-a_{\rho_0})=\gamma_{\rho_0}$, and so
$$e\ =\ \ddeg_{\ges \beta_\sigma} P_{+b_\sigma}\ \les\ \ddeg_{\ges \gamma_{\rho_0}} P_{+a_{\rho_0}}\ =\ d,$$
by \cite[Corollary~6.6.12]{adamtt}.
By symmetry, we also have $d \les e$, so $d=e$.
\end{proof}

\noindent
As in \cite[\S2]{maximext} and \cite[\S11.2]{adamtt}, we associate to each pc-sequence $(a_\rho)$ in $K$ its \emph{cut} in $K$, denoted by $c_K(a_\rho)$, such that if $(b_\sigma)$ is a pc-sequence in $K$, then
$$c_K(a_\rho) = c_K(b_\sigma)\ \iff\ (b_\sigma)\ \text{is equivalent to}\ (a_\rho).$$
Below, $\bm a  \coloneqq  c_K(a_\rho)$.
Note that $c_K(a_\rho+y)$ for $y \in K$ depends only on $\bm a$ and $y$, so we let $\bm a+y$ denote $c_K(a_\rho+y)$. Similarly, $c_K(a_\rho y)$ for $y \in K^\times$ depends only on $\bm a$ and $y$, so we let $\bm a\cdot y$ denote $c_K(a_\rho y)$.
\begin{defn}
The \emph{dominant degree of $P$ in the cut of $(a_\rho)$}, denoted by $\ddeg_{\bm a} P$, is the natural number $d\big(P,(a_\rho)\big)$ from the previous lemma.
\end{defn}

\begin{lem}\label{ddegbasic}
The dominant degree in a cut has the following properties:
\begin{enumerate}
\item $\ddeg_{\bm a} P \les \deg P$;
\item $\ddeg_{\bm a} P_{+y} = \ddeg_{\bm a+y} P$ for $y \in K$;
\item if $y \in K$ and $vy$ is in the width of $(a_\rho)$, then $\ddeg_{\bm a} P_{+y} = \ddeg_{\bm a} P$;
\item $\ddeg_{\bm a} P_{\times y} = \ddeg_{\bm a\cdot y} P$ for $y \in K^\times$;
\item if $Q \in K\{Y\}^{\neq}$, then $\ddeg_{\bm a} PQ=\ddeg_{\bm a}P + \ddeg_{\bm a} Q$;
\item if $P(\ell)=0$ for some pseudolimit $\ell$ of $(a_\rho)$ in an extension of $K$, then $\ddeg_{\bm a} P \ges 1$;
\item if $L$ is an extension of $K$, then $\ddeg_{\bm a} P=\ddeg_{\bm a_L} P$, where $\bm a_L=c_L(a_\rho)$.
\end{enumerate}
\end{lem}
\begin{proof}
Items (1), (2), (4), and (5) follow routinely from basic facts about dominant degree, (3) follows from (2), and (7) is obvious.  

For (6), let $\ell$ be a pseudolimit of $(a_\rho)$ in an extension of $K$ with $P(\ell)=0$. Take $\rho_0$ such that, for all $\rho \ges \rho_0$, $v(\ell-a_\rho)=\gamma_\rho \in \Gamma$ and 
$d\big(P,(a_\rho)\big)=\ddeg_{\ges \gamma_\rho}P_{+a_\rho}$. Let $\rho\ges \rho_0$ and set $Q=P_{+a_\rho}$. Then $Q(\ell-a_\rho)=0$, so $\ddeg Q_{\times g} \ges 1$ for any $g \in K$ with $vg = \gamma_\rho$, by Lemma~\ref{ddegroot}.
\end{proof}

\subsection*{Coarsening and specialization}\label{subsec:coarsespec}
Coarsening and specializing are central to the arguments in this paper, so we review the definitions here.
Details and proofs can be found in \cite[\S3.4]{adamtt}.

Let $\Delta \neq \{0\}$ be a proper convex subgroup of $\Gamma$.
Then we have another valuation on $K$, 
\begin{align*}
v_\Delta \colon K^\x &\to \Gamma/\Delta\\
a &\mapsto v(a)+\Delta.
\end{align*}
We denote $K$ with this valuation by $K_\Delta$ and call it the \emph{coarsening of $K$ by $\Delta$}.
Set $\dot v \coloneqq v_\Delta$, $\dot\Gamma \coloneqq \Gamma/\Delta$, and $\dot\gamma \coloneqq \gamma+\Delta$, so if $v(a)=\gamma$, then $\dot v(a)=\dot \gamma$.
The valuation ring of $K_\Delta$ is
$$\dot{\ca O}\ =\ \{a\in K:\ \dot v(a) \ges 0\}\ =\ \{a\in K:\  v(a) \ges \delta \text{ for some } \delta \in \Delta\}\ \supseteq\ \ca O$$
with maximal ideal
$$\dot\cao\ =\ \{a\in K:\ \dot v(a)>0\}\ =\ \{a\in K:\  v(a) > \Delta\}\ \subseteq\ \cao.$$
With the same derivation, $\der$, $K_\Delta$ is a valued differential field.
By \cite[Corollary~4.4.4]{adamtt}, $\der\dot\cao \subseteq \dot\cao$, so $K_\Delta$ has small derivation.
Of course, $\dot\Gamma \neq \{0\}$, so $K_\Delta$ satisfies our standing assumptions.

Its residue field, $\dot K=\dot{\ca O}/\dot\cao$, is also a valued field, called the \emph{specialization of $K$ to $\Delta$}.
For $a \in \dot{\ca O}$, let $\dot a=a+\dot\cao$.
Then the valuation $v \colon \dot{K}^\x \to \Delta$ is defined by $v(\dot a)=v(a)$, if $a \in \dot{\ca O}\setminus\dot\cao$.
Note that $\dot K$ is also a differential field because $K_\Delta$ has small derivation, and $\dot K$ has small derivation because $K$ does.
As $\Delta \neq \{0\}$, $\dot K$ also satisfies our standing assumptions.

Some pc-sequences in $K$ remain pc-sequences after coarsening or specializing. To discuss this,
we say $(a_\rho)$ is \emph{$\Delta$-fluent} if for some index $\rho_0$ we have $\gamma_{\rho'}-\gamma_\rho > \Delta$, for all $\rho'> \rho > \rho_0$; in that case $(a_\rho)$ is still a pc-sequence in $K_\Delta$.
We say $(a_\rho)$ is \emph{$\Delta$-jammed} if for some index $\rho_0$ we have $\gamma_{\rho'}-\gamma_\rho\in \Delta$, for all $\rho'> \rho > \rho_0$.
If $a_\rho \in \dot{\ca O}$ and $\gamma_\rho \in \Delta$, eventually, then $(a_{\rho})$ is $\Delta$-jammed and
$(\dot{a}_\rho)$ is a pc-sequence in $\dot K$, where by convention
we drop the indices $\rho$ for which $a_{\rho}\notin \dot{\ca O}$.
If $(a_\rho)$ is not $\Delta$-jammed, then it has a $\Delta$-fluent cofinal subsequence.

Let $\ddeg^\Delta P$ be the dominant degree of $P$ with respect to the valuation $\dot v=v_\Delta$ on $K$.
Here is how the dominant degree in a cut behaves under coarsening:

\begin{lem}\label{ddegcoarsen}
Suppose $(a_\rho)$ is $\Delta$-fluent.
Let $\bm a=c_K(a_\rho)$ and $\bm a_\Delta=c_{K_\Delta}(a_\rho)$.
Then $$\ddeg_{\bm a} P\ \les\ \ddeg_{\bm a_\Delta}^\Delta P.$$
\end{lem}
\begin{proof} For $g\in K^\times$ with $vg= \gamma_{\rho}$ and
$Q \coloneqq  P_{+a_{\rho}}$ we have
$$
\ddeg_{\ges \dot{\gamma}_\rho}^\Delta Q\ =\ \ddeg^\Delta Q_{\times g}\ \ges\ \ddeg Q_{\times g}\ =\  \ddeg_{\ges \gamma_\rho} Q.$$
It follows that $\ddeg_{\bm a} P \les \ddeg^\Delta_{\bm a_\Delta} P$.
\end{proof}

\begin{lem}\label{ddegcoarsenconj}
Suppose $P \in \dot{\ca O}\{Y\} \setminus \dot{\cao}\{Y\}$, $b \in \dot{\ca O}$, and $h \in \dot{\ca O} \setminus \dot{\cao}$.
Then
$$\ddeg \dot{P}_{+\dot{b}}\ =\ \ddeg P_{+b}\ \text{and}\ \ddeg \dot{P}_{\times\dot{h}}\ =\ \ddeg P_{\times h}.$$
\end{lem}
\begin{proof}
By \cite[Lemma~4.3.1]{adamtt}, we have $\dot{P}_{+\dot{b}} = \dot{(P_{+b})}$ and $\dot{P}_{\times\dot{h}} = \dot{(P_{\times h})}$.
It remains to note that $\ddeg Q = \ddeg \dot{Q}$, for all $Q \in \dot{\ca O}\{Y\} \setminus \dot{\cao}\{Y\}$.
\end{proof}

\begin{cor}\label{ddegspecialize}
Suppose $P \in \dot{\ca O}\{Y\} \setminus \dot\cao\{Y\}$, and $a_\rho \in \dot{\ca O}$, $\gamma_\rho \in \Delta$, eventually.
Let $\bm a=c_K(a_\rho)$ and $\dot{\bm a}=c_{\dot K}(\dot{a}_\rho)$.
Then $\ddeg_{\bm a} P = \ddeg_{\dot{\bm a}} \dot{P}.$
\end{cor}
\begin{proof}
For $g \in K^\times$ with $vg=\gamma_\rho\in \Delta$ we have
$$\ddeg_{\ges \gamma_\rho} \dot{P}_{+\dot{a}_\rho}\ =\ \ddeg \dot{P}_{+\dot{a}_\rho, \times \dot{g}}\  \text{ and }\
\ddeg_{\ges \gamma_\rho} P_{+a_\rho}\ =\ \ddeg P_{+a_\rho, \times g},$$
so the desired result follows from Lemma~\ref{ddegcoarsenconj}.
\end{proof}

\section{The differential-henselian configuration property}\label{sec:dhc}
In this section, we assume that the induced derivation on $\bm k$ is nontrivial.
This makes available tools from \cite[\S6.8 and \S6.9]{adamtt} on constructing immediate extensions, which are fundamental to our results.
We introduce here the \emph{differential-henselian configuration property}.
In \cite[Proposition~7.4.1]{adamtt}, the authors proved that monotone valued differential fields have this property, from which they deduced the uniqueness of maximal immediate extensions of monotone fields, and of $\d$-algebraically maximal $\d$-algebraic immediate extensions of monotone fields.
In \S\ref{sec:mainresults}, we show that those results depend only on the differential-henselian configuration property, not on monotonicity.

The goal then of this section is to prove that a valued differential field with value group of finite rank has the differential-henselian configuration property, in Corollary~\ref{dhcfinrank}.
Corollary~\ref{dhcfinrankunion} extends this to valued differential fields whose value group is the union of its finite rank convex subgroups.
The theorems in the introduction then follow from the results of \S4 combined with Corollary~\ref{dhcfinrankunion}.

Recall that a pc-sequence in $K$ is \emph{divergent} if it has no pseudolimit in $K$.

\begin{defn}
We say $K$ has the \emph{differential-henselian configuration property} (\emph{dh-configuration property} for short) if, for every divergent pc-sequence $(a_\rho)$ in $K$ with minimal differential polynomial $G(Y)$ over $K$, we have $\ddeg_{\bm a} G=1$.
\end{defn}

\noindent
Rephrasing \cite[Proposition~7.4.1]{adamtt} gives that monotone $K$ have the dh-configuration property. There, the assumption on $G$ was weaker, but this form was all that was necessary for the consequences mentioned above.
If $\Gamma$ has finite rank, then we call the number of nontrivial convex subgroups its \emph{rank}.
By \cite[Corollary~6.1.2]{adamtt}, if $\Gamma$ has rank 1, then $K$ is monotone (assuming small derivation, as we are).
Thus if $\Gamma$ has rank 1, then $K$ has the dh-configuration property.
This will be the base case for an inductive proof of Corollary~\ref{dhcfinrank}.

And now, we examine how the dh-configuration property relates to coarsening and specialization.
To distinguish between pseudoconvergence in $K$ and in a coarsening of $K$ with valuation $\dot v$, we write $\pconv_v$ for the former and $\pconv_{\dot v}$ for the latter.
We use $\pconv$ for pseudoconvergence in both $K$ and specializations of $K$.

\begin{prop}\label{dhcglue}
Let $\Delta \neq \{0\}$ be a proper convex subgroup of $\Gamma$.
Let $K_\Delta$ be the coarsening of $K$ by $\Delta$ and $\dot K$ be the specialization of $K$ to $\Delta$.
Suppose that $K_\Delta$ and $\dot K$ have the dh-configuration property.
Then so does $K$.
\end{prop}
\begin{proof}
Let $(a_\rho)$ be a pc-sequence in $K$ with minimal differential polynomial $G(Y)$ over $K$.
We need to show that $\ddeg_{\bm a} G = 1$. By the definition of {\em minimal differential polynomial\/} in \cite[p.\ 227]{adamtt} we can
replace $(a_\rho)$ by an equivalent pc-sequence to arrange that $G(a_\rho)\pconv 0$. 

At this point we distinguish between 
Case 1 and Case 2 below, but first we indicate a construction that
is needed in both cases and which depends only on the assumptions and arrangements that are now in place.
Using \cite[Lemma~6.9.3]{adamtt} we obtain a pseudolimit $a$ of $(a_\rho)$ in an immediate extension $K\langle a \rangle$ of $K$ with $G(a)=0$.
Hence $\ddeg_{\bm a} G \ges 1$ by Lemma~\ref{ddegbasic} (6), so it is enough to show that $\ddeg_{\bm a} G \les 1$.
Let $P=G_{+a}$, so $\deg P = \deg G$ by \cite[Corollary~4.3.2]{adamtt}. Set $\gamma_{\rho} \coloneqq v(a_{\rho+1}-a_{\rho})$.  

By \cite[Lemma~6.8.1]{adamtt} and the remarks after it,
we have a pc-sequence $(b_\rho)$ equivalent to $(a_\rho)$ such that:
\renewcommand{\labelenumi}{(\roman{enumi})}
\begin{enumerate}
\item $G(b_\rho) \pconv 0$,
\item $v(b_{\rho+1}-b_\rho)=v(b_\rho-a)=\gamma_\rho$, eventually, and
\item $v\big(G(b_\rho)\big)=v\big(P(b_\rho-a)\big)=v_{P}(\gamma_\rho)$, eventually.
\end{enumerate}
\renewcommand{\labelenumi}{(\arabic{enumi})}
Then $\bm a=c_K(b_\rho)$, as $(b_\rho)$ is equivalent to $(a_\rho)$.  
From \cite[Corollary~6.1.10]{adamtt} and $P(0)=0$ we obtain an $e$ with $1 \les e \les \deg P$ such that $P_e \neq 0$ and
$$v\big(G(b_\rho)\big)\ =\ v_P(\gamma_\rho)\ =\ \min_{d} v_{P_d}(\gamma_\rho)\ =\ v_{P_e}(\gamma_\rho),\ \text{eventually}.$$
Then \cite[Corollary~6.1.3]{adamtt} gives, for all sufficiently large $\rho$ and all $\rho'>\rho$,
\begin{equation}\tag{$*$}
v\big(G(b_{\rho'})\big)-v\big(G(b_{\rho})\big)\ =\ v_{P_e}(\gamma_{\rho'}) - v_{P_e}(\gamma_\rho)\ =\ e(\gamma_{\rho'}-\gamma_\rho) + o(\gamma_{\rho'}-\gamma_\rho).
\end{equation}

\medskip\noindent
{\em Case 1:\ $(a_\rho)$ is not $\Delta$-jammed}. 
Then $(a_\rho)$ has a $\Delta$-fluent cofinal subsequence. 
Replacing $(a_\rho)$ by such a subsequence we arrange that 
$(a_\rho)$ is $\Delta$-fluent, preserving $G(a_{\rho})\leadsto 0$. Next we do the construction above of 
$a, (b_{\rho}), P, e$.  Note that then
$(b_\rho)$ is also $\Delta$-fluent by (ii).

We claim that $\big(G(b_\rho)\big)$ is $\Delta$-fluent, that is, for all sufficiently large $\rho$ and all $\rho'>\rho$,
$$v\big(G(b_{\rho'+1})-G(b_{\rho'})\big)-v\big(G(b_{\rho+1})-G(b_{\rho})\big)\ >\ \Delta.$$
Since $G(b_\rho) \pconv 0$, $v\big(G(b_\rho)\big)$ is eventually strictly increasing,
and thus $v\big(G(b_{\rho+1})-G(b_{\rho})\big)=v\big(G(b_\rho)\big)$, eventually.
Hence it is enough to show that, for all sufficiently large
$\rho$ and all $\rho'>\rho$,
$$v\big(G(b_{\rho'})\big)-v\big(G(b_{\rho})\big)\ >\ \Delta.$$
This inequality holds by $(*)$, since $\gamma_{\rho'}-\gamma_\rho > \Delta$, eventually. 
So $\big(G(b_\rho)\big)$ is $\Delta$-fluent and $G(b_\rho) \pconv_{\dot{v}} 0$.

Next we show that $G$ remains a minimal differential polynomial of $(b_\rho)$ over $K_\Delta$.
So let $(e_\lambda)$ be a pc-sequence in $K_\Delta$ that is equivalent to $(a_\rho)$ (with respect to $\dot v$) and suppose $H \in K\{Y\}$ is such that $H(e_\lambda) \pconv_{\dot v} 0$.
Since $\dot{v}$ is a coarsening of $v$, \cite[Lemmas 2.2.21 and 2.2.17 (iii)]{adamtt} give that $(e_\lambda)$ is a pc-sequence in $K$ that is equivalent to $(b_\rho)$ (with respect to $v$).
From $H(e_\lambda) \pconv_{\dot v} 0$ we get $H(e_\lambda) \pconv_v 0$, and hence $c(H) \ges c(G)$. As $K_{\Delta}$ has the dh-configuration property, we obtain
 $\ddeg^\Delta_{\bm a_\Delta} G=1$, so $\ddeg_{\bm a} G \les 1$ by Lemma~\ref{ddegcoarsen}, as desired.

\medskip\noindent
{\em Case 2:\ $(a_\rho)$ is $\Delta$-jammed}.
Take an index $\rho_0$ so large that, for all $\rho'>\rho \ges \rho_0$, 
$$\gamma_\rho \in \Gamma,\quad v(a_{\rho'}-a_{\rho})=\gamma_{\rho},\quad \gamma_{\rho'}>\gamma_\rho, \text{ and }\gamma_\rho-\gamma_{\rho_0} \in \Delta.$$
Take $g \in K$ with $vg=\gamma_{\rho_0}$.
Then, replacing $(a_\rho)$ by $((a_\rho-a_{\rho_0})/g)$ and $G$ by $G_{+a_{\rho_0}, \times g}$, we arrange that $v(a_\rho)=0$ and $\gamma_\rho \in \Delta^{>}$, eventually, preserving $G(a_{\rho})\leadsto 0$. 
This is possible by Lemma~\ref{ddegbasic} (2) and (4).
By scaling $G$ we also arrange $v(G) =0\in \Delta$.
At this point we do the construction above 
of $a$, $(b_\rho)$, $P$, $e$. From $v(a)=0$ and (ii) we get
$v(b_\rho)=0$, eventually. 

We claim that $\dot G(\dot b_\rho)$ is a pc-sequence in $\dot K$ with $\dot G(\dot b_\rho) \pconv 0$.
First, by \cite[Lemma~4.5.1]{adamtt}, $v(P)=0$. This yields $v(P_e) \in \Delta$: otherwise, $v(P_e)>\Delta$ and so by taking $d \neq e$ with $v(P_d)=0$ we obtain $v_{P_e}(\gamma_\rho)>v_{P_d}(\gamma_\rho)$, eventually, by \cite[Corollary~6.1.5]{adamtt} and
because $\gamma_\rho \in \Delta$, eventually; but this contradicts the choice of $e$. Next, from \cite[Corollary~6.1.3]{adamtt} and $v(P_e)\in \Delta$ we get
$$ v\big(G(b_\rho)\big)\ =\ v_{P_e}(\gamma_\rho)\ =\ v(P_e)+e\gamma_\rho + o(\gamma_\rho)\ \in\ \Delta,\ \text{eventually},$$
so $v\big(\dot{G}(\dot{b}_\rho)\big)=v(P_e)+e\gamma_\rho + o(\gamma_\rho)$,
eventually, and thus $\dot{G}(\dot{b}_\rho)$ is a pc-sequence in $\dot K$ with $\dot{G}(\dot{b}_\rho) \pconv 0$.

Finally, we show that $\dot G$ is a minimal differential polynomial of $(\dot a_\rho)$ over $\dot K$. 
So let $(e_{\lambda})$ be a well-indexed sequence in $\dot{\ca O}$
such that
$(\dot{e}_\lambda)$ is a pc-sequence in $\dot K$ equivalent to $(\dot{a}_\rho)$, and thus to $(\dot{b}_\rho)$, and let $H \in \dot{\ca O}\{Y\}$ be such that $\dot H \in \dot K\{Y\}^{\ne}$,
$\dot{H}(\dot{e}_\lambda) \pconv 0$, and $c(H)=c(\dot H)$.  Then \cite[Lemma~3.4.1]{adamtt} gives $H(e_\lambda) \pconv 0$.
Moreover, $(e_\lambda)$ is equivalent to $(b_\rho)$, so $c(H) \ges c(G)$, and thus $c(\dot H) \ges c(\dot G)$.
The hypothesis on $\dot{K}$ therefore yields $\ddeg_{\dot{\bm a}} \dot G = 1$, and so $\ddeg_{\bm a} G = 1$ by Corollary~\ref{ddegspecialize}.
\end{proof}

\noindent
Note that the part of the proof preceding Case 1 does not use the assumptions on $K_{\Delta}$ and
$\dot{K}$, that Case 1 only uses the assumption on $K_{\Delta}$, and Case 2 only the assumption on $\dot{K}$. 

\begin{cor}\label{dhcfinrank}
Suppose $\Gamma$ has finite rank.
Then $K$ has the dh-configuration property.
\end{cor}
\begin{proof}
By induction on the rank of $\Gamma$.
If $\Gamma$ has rank 1, then $K$ is monotone, so has the dh-configuration property.
If $\Gamma$ has rank $n>1$, then it has a proper convex subgroup $\Delta \neq \{0\}$. 
Both $\Gamma/\Delta$ and $\Delta$ have rank $<n$, so the coarsening of $K$ by $\Delta$ and the specialization of $K$ to $\Delta$ have the dh-configuration property by the induction hypothesis.
Then Proposition~\ref{dhcglue} gives the result.
\end{proof}

\begin{defn}
An ordered abelian group $G$ has the \emph{differential-henselian configuration property} (\emph{dh-configuration property} for short) if every valued differential field with small derivation, nontrivial induced derivation on its residue field, and value group $G$ 
has the dh-configuration property.
\end{defn}

\noindent
Thus any $G$ of finite rank has the dh-configuration property by Corollary~\ref{dhcfinrank}. This property is inherited by ``convex'' unions, but for that we need the following:

\begin{lem}\label{dhccauchy}
Let $(a_\rho)$ be a pc-sequence in $K$ of width $\{\infty\}$ and 
$G(Y)$ a minimal differential polynomial of $(a_{\rho})$ over $K$.
Then $\ddeg_{\bm a} G=1.$
\end{lem}
\begin{proof} As in the proof of Proposition~\ref{dhcglue} we arrange
$G(a_\rho)\pconv 0$, take a pseudolimit $a$ of $(a_\rho)$ in an immediate extension 
of $K$ with $G(a)=0$,
and set
$P \coloneqq G_{+a}$ (with $\deg P = \deg G$), and $\gamma_{\rho} \coloneqq  v(a_{\rho+1}-a_{\rho})$. We also arrange that $(\gamma_{\rho})$
is strictly increasing and $v(a-a_{\rho})=\gamma_{\rho}$ for all $\rho$. 
We claim that $P_1\ne 0$. To prove this claim, let $r$ be the order of $G$. Then $\partial G/\partial Y^{(r)} \neq 0$. Hence $\big(\partial G/\partial Y^{(r)}\big)(a) \neq 0$, since otherwise \cite[Lemma~6.8.1]{adamtt} would give a pc-sequence $(b_\rho)$ in $K$ equivalent to $(a_\rho)$ such that
$\frac{\partial G}{\partial Y^{(r)}}(b_\rho)\ \pconv\ 0$,
contradicting the minimality of $G$.
In view of $P_1 = \sum_{i=0}^{r} \frac{\partial G}{\partial Y^{(i)}}(a) \cdot Y^{(i)}$, this gives $P_1 \neq 0$. 
 
For $\ddeg_{\bm a} G=1$, it is enough by \cite[Corollary~6.6.6]{adamtt} that $\ddeg G_{+a, \x g}=1$ for some $\rho$ and $g \in K^\x$ with $v(g) = \gamma_\rho$.
As $(P_d)_{\x g}=(G_{+a, \x g})_d$ for $g\in K^\times$, it suffices to show that for all $d>1$ with $P_d \neq 0$ we have
$v_{P_d}(\gamma_\rho)>v_{P_1}(\gamma_\rho)$, eventually.
Now $P_1 \neq 0$, so if $d>1$ and $P_d \neq 0$, then
$$v_{P_d}(\gamma_\rho)-v_{P_1}(\gamma_\rho)\ =\ v(P_d)-v(P_1)+(d-1)\gamma_\rho + o(\gamma_\rho),$$
by \cite[Corollary~6.1.5]{adamtt}, and thus $v_{P_d}(\gamma_\rho)>v_{P_1}(\gamma_\rho)$, eventually, 
since $(\gamma_\rho)$ is cofinal in $\Gamma$.
\end{proof}

\begin{lem}\label{dhcunionlem}
Suppose $\Gamma$ is a union of convex subgroups with the dh-configuration property.
Then $\Gamma$ has the dh-configuration property.
\end{lem}
\begin{proof}
Let $\Gamma = \bigcup_{i \in I} \Delta_i$ for some index set $I$, where each $\Delta_i$ is a nontrivial convex subgroup of $\Gamma$ with the dh-configuration property.
The case that $\Gamma=\Delta_i$ for some $i \in I$ is trivial, so suppose $\Gamma \neq \Delta_i$ for each $i \in I$.
Let $(a_\rho)$ be a pc-sequence in $K$ with a minimal differential polynomial $G(Y)$ over $K$.
Set $\bm a \coloneqq c_K(a_\rho)$ and $\gamma_\rho \coloneqq v(a_{\rho+1}-a_\rho)$;
we can assume that $(\gamma_{\rho})$ is strictly increasing.

If $(a_\rho)$ is $\Delta_i$-jammed for some $i$, then $\ddeg_{\bm a} G=1$ by the argument from Case 2 in the proof of Proposition~\ref{dhcglue}. 
If $(a_\rho)$ is not $\Delta_i$-jammed for any $i$, then 
$(\gamma_\rho)$ is cofinal in $\Gamma$, so $\ddeg_{\bm a} G=1$ by Lemma~\ref{dhccauchy}.
\end{proof}

\begin{cor}\label{dhcfinrankunion}
If $\Gamma$ is the union of its finite rank convex subgroups, then $\Gamma$ has the dh-configuration property.\footnote{Moreover, Allen Gehret has pointed out that if the convex subgroups of $\Gamma$ are well-ordered by inclusion, then it follows from Proposition~\ref{dhcglue} and Lemma~\ref{dhcunionlem} by transfinite induction that $\Gamma$ has the dh-configuration property.}
\end{cor}

\section{Main results}\label{sec:mainresults}
We now show how the desired results follow from the dh-configuration property, without other assumptions on $\Gamma$.
Theorem~\ref{dhtodalgmax} is stated for $K$ with the dh-configuration property, while Theorems \ref{maximmed} and \ref{dhensel} are stated for $\Gamma$ with the dh-configuration property.
(In fact, all that is needed in the latter two theorems is that
all immediate extensions of $K$ have the dh-configuration property.)
The results claimed in the introduction then follow in view of Corollary~\ref{dhcfinrankunion}.
We first use the dh-configuration property to find pseudolimits of pc-sequences that are also zeroes of their minimal differential-polynomials.
The argument is the same as in \cite[Lemma~7.4.2]{adamtt}.

\begin{lem}\label{dalgmaxplimroot}
Suppose $K$ has the dh-configuration property.
Let $(a_\rho)$ be a divergent pc-sequence in $K$ with minimal differential polynomial $G(Y)$ over $K$.
Let $L$ be a $\d$-algebraically maximal extension of $K$ such that $\bm k_L$ is linearly surjective.
Then $a_\rho \pconv b$ and $G(b)=0$ for some $b\in L$.
\end{lem}
\begin{proof}
Note that $L$ is $\d$-henselian by \cite[Theorem~7.0.1]{adamtt}.
Since $L$ is $\d$-algebraically maximal and the derivation of $\bm k_L$ is nontrivial, every pc-sequence in $L$ of $\d$-algebraic type over $L$ has a pseudolimit in $L$ by \cite[Lemma~6.9.3]{adamtt}.
Thus we  get $a \in L\setminus K$ with $a_\rho \pconv a$. Passing to
an equivalent pc-sequence we arrange that $G(a_\rho) \pconv 0$.
With $\gamma_\rho =v(a_{\rho+1}-a_{\rho})=v(a-a_\rho)$, eventually, the dh-configuration property gives $g_\rho \in K$ with $v(g_\rho)=\gamma_\rho$ and $\ddeg G_{+a_\rho, \times g_\rho}=1$, eventually.
By \cite[Corollary~6.6.6]{adamtt}, $\ddeg G_{+a, \times g_\rho}=1$, eventually.
We have $G(a+Y)=G(a)+A(Y)+R(Y)$ where $A$ is linear and homogeneous and all monomials in $R$ have degree $\ges 2$, and so
$$G_{+a, \times g_{\rho}}(Y)\ =\ G(a)+A_{\times g_{\rho}}(Y) + R_{\times g_{\rho}}(Y).$$ 
Now $\ddeg G_{+a,\times g_{\rho}}=1$ eventually, so
$v(G(a)) \ges v_A(\gamma_\rho)<v_R(\gamma_\rho)$ eventually.
With $v(G,a)$ as in \cite[\S7.3]{adamtt} we get $v(G,a)>v(a-a_\rho)$, eventually.
Then \cite[Lemma~7.3.1]{adamtt} gives $b \in L$ with $v_L(a-b)=v(G,a)$ and $G(b)=0$, so $v_L(a-b)>v(a-a_\rho)$ eventually. Thus $a_\rho \pconv b$.
\end{proof}

\begin{thm}\label{maximmed}
Suppose $\bm k$ is linearly surjective and $\Gamma$ has the dh-configuration property.
Then any two maximal immediate extensions of $K$ are isomorphic over $K$.
Also, any two $\d$-algebraically maximal $\d$-algebraic immediate extensions of $K$ are isomorphic over $K$.
\end{thm}
\begin{proof}
Let $L_0$ and $L_1$ be maximal immediate extensions of $K$.
By Zorn we have a maximal isomorphism $\mu \colon F_0 \cong_K F_1$ between valued differential subfields $F_i\supseteq K$ of $L_i$ for $i=0,1$, where ``maximal'' means that $\mu$ does not extend to
an isomorphism between strictly larger such valued differential
subfields. 
Suppose towards a contradiction that $F_0 \neq L_0$ (equivalently, $F_1 \neq L_1$).
Then $F_0$ is not spherically complete, so we have a divergent pc-sequence $(a_\rho)$ in $F_0$.

Suppose $(a_\rho)$ is of $\d$-transcendental type over $F_0$. The spherical completeness of $L_0$ and $L_1$ then
yields $f_0 \in L_0$ and $f_1 \in L_1$
such that $a_\rho \pconv f_0$ and $\mu(a_\rho) \pconv f_1$. 
Hence by \cite[Lemma~6.9.1]{adamtt} we obtain an isomorphism $F_0\langle f_0 \rangle \cong F_1\langle f_1 \rangle$ extending
$\mu$, contradicting the maximality of $\mu$.

Suppose $(a_\rho)$ is of $\d$-algebraic type over $F_0$, with minimal differential polynomial $G$ over $F_0$.
Then Lemma~\ref{dalgmaxplimroot} gives $f_0 \in L_0$ with $a_\rho \pconv f_0$ and $G(f_0)=0$, and $f_1 \in L_1$ with $\mu(a_\rho) \pconv f_1$ and $G^\mu(f_1)=0$.
Now \cite[Lemma~6.9.3]{adamtt} gives an isomorphism $F_0\langle f_0 \rangle \cong F_1\langle f_1 \rangle$ extending $\mu$,
and we have again a contradiction. Thus $F_0=L_0$ and
hence $F_1=L_1$ as well.

The proof of the second statement is the same, using only \cite[Lemma~6.9.3]{adamtt}.
\end{proof}

\noindent
In the case of few constants, we have the following additional results.
The first has essentially the same proof as \cite[Theorem~7.0.3]{adamtt}.

\begin{thm}\label{dhtodalgmax}
Suppose $K$ has the dh-configuration property.
Let $L$ be a $\d$-henselian extension of $K$ with few constants.
Let $(a_\rho)$ be a pc-sequence in $K$ with minimal differential polynomial $G(Y)$ over $K$.
Then $(a_\rho)$ has a pseudolimit in $L$.
In particular, if $K$ itself is $\d$-henselian and has few constants, then it is $\d$-algebraically maximal.
\end{thm}
\begin{proof}
Towards a contradiction, assume that $(a_\rho)$ is divergent in $L$. Then \cite[Lemma~6.9.3]{adamtt} shows that $G$ has order 
$\ges 1$ (since $L$ is henselian) and provides a proper immediate extension $L\langle a \rangle$ of $L$ with $a_\rho \pconv a$. 
Replacing $(a_\rho)$ by an equivalent pc-sequence in $K$, we arrange that $G(a_\rho) \pconv 0$.

By the dh-configuration property, $\ddeg_{\bm a} G=1$.
Taking $g_\rho \in K$ with $v(g_\rho)=\gamma_\rho$ we have
$\ddeg G_{+a_\rho, \times g_\rho} = 1$, eventually. By removing some initial terms of the sequence, we arrange that this holds for all $\rho$ and that $v(a-a_\rho)=\gamma_\rho$ for all $\rho$.
By $\d$-henselianity and \cite[Lemma~7.1.1]{adamtt}, we have $z_\rho \in L$ with $G(z_\rho)=0$ and $a_\rho - z_\rho \preccurlyeq g_\rho$.
From $a-a_\rho \asymp g_\rho$ we get $a-z_\rho \preccurlyeq g_\rho$.

Let $r\ges 1$ be the order of $G$.
By \cite[Lemma~2.2.19]{adamtt}, $(\gamma_\rho)$ is cofinal in $v(a-K)$, so there are indices $\rho_0<\rho_1<\dots<\rho_{r+1}$ such that $a-z_{\rho_j} \prec a-z_{\rho_i}$, for $0 \les i<j \les r+1$.
Then
$$z_{\rho_i}-z_{\rho_{i-1}}\ \asymp\ a-z_{\rho_{i-1}}\ \succ\ a-z_{\rho_i}\ \asymp\ z_{\rho_{i+1}}-z_{\rho_i},\ \text{for}\ 1 \les i \les r.$$
We have $a-z_{\rho_{r+1}} \prec a-z_{\rho_0} \prece g_{\rho_0}$, so $z_{\rho_0}-z_{\rho_{r+1}} \prece g_{\rho_0}$, and thus also $a_{\rho_0}-z_{\rho_{r+1}} \prece g_{\rho_0}$.
Hence
$$\ddeg G_{+z_{\rho_{r+1}}, \x g_{\rho_0}}\ =\ \ddeg G_{+z_{\rho_0}, \x g_{\rho_0}}\ =\ \ddeg G_{+a_{\rho_0}, \times g_{\rho_0}}\ =\ 1$$
by \cite[Corollary~6.6.6]{adamtt}.
Thus, with $z_{\rho_i}$ in the role of $y_i$, for $0 \les i \les r+1$, and $g_{\rho_0}$ in the role of $g$, we have reached a contradiction with \cite[Lemma~7.5.5]{adamtt}.
\end{proof}

\noindent
Next a result playing the same role in Theorem~\ref{dhensel} as Lemma~\ref{dalgmaxplimroot} played in Theorem~\ref{maximmed}.

\begin{cor}\label{dhplimroot}
If $K$ has the dh-configuration property, $L$ is a $\d$-henselian extension of $K$ with few constants, and
$(a_\rho)$ is a divergent pc-sequence in $K$ with minimal differential polynomial $G(Y)$ over $K$, then
$a_\rho \pconv b$ and $G(b)=0$ for some $b\in L$.
\end{cor}
\begin{proof}
Follow the proof of Lemma~\ref{dalgmaxplimroot}, invoking Theorem~\ref{dhtodalgmax} instead of $\d$-algebraic maximality.
\end{proof}

\noindent
To our knowledge, the following is the first known result about differential-henselizations.

\begin{thm}\label{dhensel}
Suppose $K$ is asymptotic, $\bm k$ is linearly surjective, and $\Gamma$ has the dh-configuration property.
Then $K$ has a $\d$-henselization, and any two $\d$-henselizations of $K$ are isomorphic over~$K$.
\end{thm}
\begin{proof}
By \cite[Corollary~9.4.11]{adamtt} we have an immediate asymptotic $\d$-henselian extension $K^{\dh}$ of $K$ that is $\d$-algebraic over $K$ and has no proper $\d$-henselian subfields containing $K$.
By Theorems \ref{dhtodalgmax} and \ref{maximmed}, there is up to isomorphism over $K$ just one such extension.

Let $L$ be an immediate $\d$-henselian extension of $K$; then $L$ is asymptotic by \cite[Lemmas 9.4.2 and 9.4.5]{adamtt}.
To see that $K^{\dh}$ embeds over $K$ into $L$, use an argument similar to that in the proof of Theorem~\ref{maximmed}, using Corollary~\ref{dhplimroot} in place of Lemma~\ref{dalgmaxplimroot}. Thus $K^{\dh}$ is a $\d$-henselization of
$K$ and any $\d$-henselization of $K$ is isomorphic over $K$ to $K^{\dh}$. 
\end{proof}

\noindent
In fact, the argument shows that $K^{\dh}$ as in the proof of Theorem~\ref{dhensel} embeds over $K$ into any (not necessarily immediate) asymptotic $\d$-henselian extension of $K$.
\begin{cor}
Suppose $K$ is asymptotic, $\bm k$ is linearly surjective, and $\Gamma$ has the dh-configuration property.
Then any immediate $\d$-henselian extension of $K$ that is $\d$-algebraic over $K$ is a $\d$-henselization of $K$.
\end{cor}
\begin{proof}
Let $K^{\dh}$ be the $\d$-henselization of $K$ from the proof of Theorem~\ref{dhensel}. Then
$K^{\dh}$ is asymptotic, so has few constants, and is thus $\d$-algebraically maximal by Theorem~\ref{dhtodalgmax}.
Hence any $K$-embedding $K^{\dh}\to L$ into an immediate $\d$-algebraic $\d$-henselian extension $L$ of $K$ is surjective.
\end{proof}

\end{document}